\documentclass[12pt,reqno]{amsart}

\usepackage{latexsym}
\usepackage{amssymb}
\usepackage{graphics}
\usepackage{graphicx}

\renewcommand{\epsilon}{\varepsilon}

\newcommand{\PP}{{\mathbb P}}

\newcommand{\R}{{\mathbb R}}
\newcommand{\C}{{\mathbb C}}

\newcommand{\CP}{\C\PP}
\renewcommand{\d}{\partial}

\newcommand{\SU}{{\operatorname{SU}}}

\renewcommand{\phi}{\varphi}

\newcommand{\ocal}{\mathcal{O}}

\newtheorem{theo}{{Theorem}}[section]
\newtheorem{cor}[theo]{{Corollary}}
\newtheorem{lem}[theo]{{Lemma}}
\newtheorem{prop}[theo]{{Proposition}}

\newenvironment{rem}{\medskip\noindent{\it Remark:\/} }{\medskip}

\newenvironment{claim}{\medskip\noindent{\it Claim:\/} }{\medskip}

\title[Critical Points of holomorphic sections]{Critical Points of holomorphic sections of line bundles and a spherical Gauss-Lucas theorem }

\author{Jingzhou Sun }
\address{Department of Mathematics, Stony Brook University, Stony Brook, NY
11777, USA} \email{jsun@math.stonybrook.edu}

\date{\today}

\begin{document}

\begin{abstract}
We study critical points of holomorphic sections of $\ocal(m)$ on $\CP^n$. For quadrics, we give a complete discription of their critical points. When $n=1$, we prove a spherical Gauss-Lucas theorem. For general situation, we prove that a general section has all its critical points isolated and non-degenerate.
\end{abstract}

\maketitle

 \tableofcontents
 \section{Introduction}
The critical points of a section $s$ of a line bundle over a complex manifold is the points where $\nabla s$ vanish, where $\nabla$ is a Hermitian connection.
The study of distributions of critical points of holomorphic sections of an ample line bundle over a K\"ahler manifold was initiated by Douglas,  Shiffman, and Zelditch in \cite{dsz1}\cite{dsz2}, one of whose motivations came from Physics. They proved general formulas for expected density of critical points of holomorphic sections of a line bundle which is not necessarily ample. In the case of $\ocal(m)$ over $\CP^n$, they were able to use the formulas they proved, with the aid of Maple, to calculate out the explicit densities of each index for $m\leq 6$. Based on their pioneering work, a further consideration is to have a better understanding of the distribution of critical points, which should be more precise than the expected density.  This consideration is one motivation of the current work.

A second motivation of this work is from \cite{s}, where the author proved a formula for expected Euler characteristic of excursion sets of random sections of an ample line bundle. The formula proved there is only valid for big enough cutting values. When trying to understand the more general situation, one has to encounter the problem of understanding the more general critical points of random sections.

An important application, which is also a motivation, of the study critical points of holomorphic sections is  in Morse theory. Because the critical points of the norm function $|s|^2$ is the union of the critical points of $s$ with the zero variety of $s$, understanding the critical points of $s$ is the same thing as understanding those of the function $|s|^2$.  In \cite{ab}, Atiyah and Bott extended the Morse theory with non-degenerate critical points replaced by non-degenerate critical manifolds, which fits exactly when one studies critical points of holomorphic sections of a line bundle, since the zero variety is in general a non-degenerate critical manifold  of the norm function $|s|^2$ of a section $s\in \textrm{H}^0(M,L)$(cf.\cite{bo}). This generalized Morse theory together with theorem \ref{theo:nondeg} provide a way to compute the middle cohomology class of a hypersurface in $\CP^{2n+1}$. 

For example, when $s$ is a general section of $\textrm{H}^0(\CP^n, \ocal(2))$, we are able to have a complete description of its critical points.
\begin{theo}\label{theo:quadric}
Let $D\subset\CP^n$ be a smooth quadratic hypersurface. Then by an action of $\SU(n+1)$, we can assume $Q$ is defined as the zero set of $s=\sum_{i=0}^n a_i Z_i^2$, with $a_i$ all positive and satisfying $a_1<a_2<\cdots<a_n$. We denote by $p_i$ the point whose only nonvanishing homogeneous coordinate is $Z_i$. 
Then the critical points of $s$ are $\{p_i|0\leq i\leq n\}$, which are all non-degenerate. Furthermore the index of $|s|^2$ at $p_i$ is $n+i$.
\end{theo}   

A direct corollary is the following well known result.
\begin{cor}\label{cor}
When $n$ is odd, the Poincare polynomial of $D$ is $\sum_{i=0}^{n-1}t^{2i}+t^{n-1}$
\end{cor}

The distribution of critical points of holomorphic sections on the Riemann sphere itself is interesting enough. Recall that the well-known Gauss-Lucus theorem on the complex plane states that the critical points of a polynomial is contained in contained in the convex hull generated by its zeroes. When one consider the direct generalization to $\CP^1$, polynomials are replaced by holomorphic sections of $\ocal(m)$, ordinary critical points replaced by critical points under the Fubini-Study connection. This is of course much more complicated. For example one can not take the convex hull of the zeroes arbitrarily, because it may be the whole $\CP^1$, in which case there is nothing to prove. By some restrictions on the locations of the zeroes,  we can prove a spherical Gauss-Lucas theorem
\begin{theo}\label{theo:gauss}
Let $s\in H^0(\CP^1, \ocal(n))$ be a nonzero section for $n\geq 2$. And suppose that the zeroes of $s$ is contained in an open hemisphere under the Fubini-Study metric. Let $P$ denote the convex hull spanned by the zeroes, and let $P_{\infty}$ denote the opposite of $P$. Then the union of the  two convex hull contains all the critical points of $s$. In other words, let $C_s$ ($C_{s,P}$, $C_{s,P_{\infty}})$ respectively) denote the set of critical points of $s$(contained in $P$ and in $P_{\infty}$ respectively), then $C_s=C_{s,P}\cup C_{s,P_{\infty}}$. Moreover, $C_{s,P_{\infty}}$ is never nonempty and if $P$ is not a point then  $C_{s,P_{\infty}}$ is  contained in the interior of  $P_{\infty}$, also $C_{s,P_{\infty}}$ contains at least one critical point of index 2 if non-degenerate. And if all the critical points of $|s|^2$ are non-degenerate, then $C_{s,P}$  is also nonempty and contained in the interior of $P$.
\end{theo} 
\begin{rem}
\begin{itemize}

\item When the zeroes lie on one minimal geodesic or even on one point, by $P$ we will just mean the closed minimal geodesic or the single point. 

\item Since every point on the Riemann sphere has a unique opposite point, namely the point whose distance to the given point is $\pi/2$ under the Fubini-Study metric. So given a subset $S$ of $\CP^1$, the opposite $S_{\infty}$ is the subset which consists of the opposite points of the points in $S$.

\item The condition of open hemisphere is necessary. For example, when $n=2$ and the two zeroes are a pair of opposite points, then all the points on the equator are critical points.

\item This result implies that there is a nontrivial correlation between the zeroes and the critical points of random sections in $\textrm{H}^0(\CP^1, \ocal(N))$. The readers are referred to \cite{ha} for a close study of this correlation in the Euclidean case.
\end{itemize}
\end{rem}

\smallskip

The idea that a generic section should have have only isolated critical points are known to \cite{dsz1}, and this in fact follows from the arguments there. However, we would like to give a different proof of this fact in the case of $\CP^n$ to illustrate a clearer picture of the bad points.

\begin{theo}\label{theo:nondeg}
In the complement of a subset of zero Lebesgue measure of $H^0(\CP^n, \ocal(m))$, every section $s$ has only non-degenerate isolated critical points. In other words, the function $|s|^2$ is a generalized Morse function.
\end{theo}

\smallskip

This article is organized as follows. In section \ref{sec:basic}, we introduce basic background of critical points of holomorphic sections line bundles. And after introducing the basic ideas in our calculations, we will prove theorem \ref{theo:nondeg}.  In section \ref{sec:quadrics}, we prove theorem \ref{theo:quadric}. Then after quickly introducing the generalized Morse theory by Atiyah and Bott, we prove corollary \ref{cor}. In section \ref{sec:gauss}, we study critical points on the Riemann sphere and prove theorem \ref{theo:gauss}.

\textbf{Acknowledgement}: The author would like to especially  thank Professor Bernard Shiffman for his continuous help, without which this work would be impossible to appear.  The author would like to thank Professor Steve Zelditch for his important encouragements. Finally the author would like to thank Renjie Feng and Yi Zhu for helpful discussions.
\section{Critical points and their non-degeneracy}\label{sec:basic}

We refer to \cite{dsz1} for more detailed introduction to critical points of holomorphic sections of a line bundle over a complex manifold.

Let $L$ be an ample line bundle over a complex manifold $X$. If $L$ is equipped with a Hermitian metric $h$, then there exists a unique Hermitian connection $\nabla:\Gamma(X,L)\rightarrow \Gamma(X, T^*\otimes L)$. If we have a local frame $e_L$, we write $|e_L|_h^2=e^{-\phi}$. Let $s\in \textrm{H}^0(X,L)$, then locally $s=fe_L$. Then $\nabla s=(df-fd\phi)\otimes e_L$. Since $T^*\otimes\C=T^{*'}\oplus T^{*''}$, the holomorphic cotangent bundle and its conjugate. We can then decompose $\nabla=\nabla'+\nabla''$ accordingly. So $\nabla' s=(\d f-f\d\phi)\otimes e_L$. Since $\nabla''s=0$ for holomorphic sections, $\nabla s=0$ is equivalent to $\nabla' s=0$.

Consider the global function $|s|^2$, which locally equals $|f|^2e^{-\phi}$. Then away from the zero locus of $s$, a point is a critical point of $|s|^2$ if and only if it is a critical point of $\log |s|^2$. Also the indices are the same at each such point.  Now  $d\log |s|^2=0$ is equivalent to $\d \log |s|^2=0$ and locally $\d \log |s|^2=\frac{\d f}{f}-\d\phi$. Therefore away from the zero locus of $s$, the critical points of $s$ and those of $\log |s|^2$ coincide. Also when the index of the latter is compounded with $(-1)^{\bullet}$, it coincide with the topological index of the former at a critical point. In this article, when we refer to the index $s$ at a point, we will mean the index $\log |s|^2$ at that point.

To compute the indices, we need to calculate the Hessian. It is shown in \cite{dsz1} that by conjugating the Hessian matrix 
\[ \left( \begin{array}{cc}
(\frac{\d^2}{\d x_j \d x_k}\log |s|^2) & (\frac{\d^2}{\d x_j \d y_k}\log |s|^2) \\
(\frac{\d^2}{\d y_j \d x_k}\log |s|^2) & (\frac{\d^2}{\d y_j \d y_k}\log |s|^2) \end{array} \right)\]

with the unitary matrix 
\[\frac{1}{\sqrt{2}}\left(\begin{array}{cc}
I & iI\\
iI & I\end{array}\right)\]

one gets 
\[ 2\left( \begin{array}{cc}
(\frac{\d^2}{\d z_j \d \bar{z}_k}\log |s|^2) & (i\frac{\d^2}{\d z_j \d z_k}\log |s|^2) \\
(-i\frac{\d^2}{\d \bar{z}_j \d \bar{z}_k}\log |s|^2) & (\frac{\d^2}{\d z_j \d \bar{z}_k}\log |s|^2) \end{array} \right)\]

We will denote the above matrix by $J_s$

From now on we will focus on the line bundles $\ocal(m)$ over $\CP^n$. Using homogeneous coordinates $[Z_0,Z_1,\cdots, Z_n]$ for $\CP^n$, we denote by $U_i$ the open affine set $Z_i\neq 0$, and by $H_i$ the complement of $U_i$, which is a hyperplane. Because of the transitive action of $SU(n+1)$ on $\CP^n$, we need only to study the situation at one point. We will fix $0=[1,0,0,\cdots,0]\in U_0$. Then in $U_0$, we use coordinates $z_i=\frac{Z_i}{Z_0}$. Thus any $s\in H^0(\CP^n, \ocal(m))$ can be written as a polynomial $f_s$ of degree at most $m$. And the Fubini-Study potential is $\phi=m\log (1+|z|^2)$, where $|z|^2=\sum_{i=1}^n |z_i|^2$. So we have
\begin{eqnarray}
\d \phi&=&\frac{\sum_{i=1}^nm\bar{z_i}d z_i}{1+|z|^2}\\
\frac{\d^2}{\d z_j \d z_k} \phi&=&-m\frac{\bar{z}_j\bar{z}_k}{1+|z|^2}\\
\frac{\d^2}{\d z_j \d \bar{z}_k} \phi&=&m\frac{\delta_{jk}(1+|z|^2)-\bar{z}_jz_k}{(1+|z|^2)^2}
\end{eqnarray} 
where by $|z|^2$ we mean $\sum_{i=1}^n|z_i|^2$, and $\delta_{jk}$ is the Kronecker delta.
\smallskip

Now we consider the linear map $\nabla' |_0: H^0(\CP^n, \ocal(m))\rightarrow T_0^{*'}\otimes \ocal(m)_0$, defined by $\nabla' |_0(s)=\nabla's(0)$.

 \begin{claim}  $\nabla' |_0$ is surjective.  
 \end{claim}
 \begin{proof}
 Since $\partial \phi(0)=0$, $\nabla's(0)=\partial f_s$. Therefore, the images of $Z_i, i\neq 0$ under  $\nabla' |_0$ span $T_0^{*'}\otimes \ocal(m)_0$. 
 \end{proof} 
So if we denote by $A_0$ the kernel of $\nabla' |_0$, then $A_0$ is a subspace of $H^0(\CP^n, \ocal(m))$ of codimension $n$. Actually one can easily see that $A_0$ is spanned by monomials $Z^\alpha$, where $\alpha=(\alpha_0,\alpha_2,\cdots, \alpha_n)$ is a multi-index satisfying $|\alpha|=m$ and $\alpha_{0}\neq m-1$.

Now let $s\in A_0$ and assume $f_s(0)=1$.  Then we have

\begin{eqnarray*}
\frac{\d^2}{\d z_j \d \bar{z}_k}\log |s|^2 (0)&=&-m\delta_{jk}\\
\frac{\d^2}{\d z_j \d z_k}\log |s|^2(0) &=&2a_{jk}\\
\frac{\d^2}{\d \bar{z}_j \d \bar{z}_k}\log |s|^2(0)&=&2\bar{a}_{jk}
\end{eqnarray*}
where $\sum_{j,k=1}^na_{jk}z_jz_k$ with $a_{jk}=a_{kj}$ is the degree 2 homogeneous part of  $f_s$.

Therefore the matrix $J_s$ at $0$ has a very simple form:
\[J_s(0)=2\left(\begin{array}{cc}
-mI & (i2a_{ij})\\
(-i2\bar{a}_{ij}) & -mI
\end{array}\right)\]

We will denote $A=(i2a_{ij})$. If we let $Q_s$ denote the following upper triangular block matrix
\[Q_s=\left(\begin{array}{cc}
I & \frac{1}{n}A\\
0 & I 
\end{array}\right)\]
then let $Q^*_s$ denote the conjugate transpose of $Q_s$, we have the following equation of matrices
\[
Q_sJ_sQ_s^*=\left(
\begin{array}{cc}
-nI+\frac{1}{n}A\bar{A} & 0\\
0 & -nI

\end{array}
\right)\]
We conclude the calculations above by the following proposition
\begin{prop}\label{theo:at 0}
$0$ is an isolated and non-degenerate critical point if and only if $-nI+\frac{1}{n}A\bar{A}$ is non singular.
\end{prop}

The remaining of this section will be devoted to the proof of theorem \ref{theo:nondeg}

\begin{proof}
When $m=1$, this statement is trivial. When $m=2$, this will become very clear after the proof of theorem \ref{theo:quadric} in section \ref{sec:quadrics}. So in the following, we assume that $m> 2$.

Let $s\in A_0$, we assume $f_s(0)=1$ and denote by $f_{s,2}$ the homogeneous part of degree $2$. Furthermore, by applying a $SU(n+1)$ action, we assume $f_{s,2}=\sum_{i=1}^na_iz_i^2$. Then the condition that $0$ is a degenerate critical point of $s$ is equivalent to the condition that $|a_i|=\frac{m}{2}$ for some $i$. From this, one sees that the degenerate condition defines real  codimension one subvariety. We denote by $B_0\subset A_p$ this subvariety.

Next, we need to consider the intersections $B_p\cap A_0$ for $p\in \CP^n$. If we can show that the union of these intersections is of Lebesque measure zero in $A_0$, this will prove the theorem. Recall that the elements in $SU(n+1)$ that fix $0$ is a subgroup which is isomorphic to $SU(n)$

If $p\in H_0$, then because of the transitive action of $SU(n)\subset SU(n+1)$, we can assume that $p=[0,1,0,\cdots,0]$. Then $A_0\cap A_p$ as a subspace of $A_0$ is defined by $Z_iZ_1^{m-1}=0$ for $i\neq 1$. So  $A_0\cap A_p$ is of codimension $n$. Also, $A_0\cap A_p$ is not contained in $B_p$, so $B_p\cap A_0$ is a codimension 1 real subvariety of $A_0\cap A_p$.

When $p\in U_0$, we want to get the same conclusion.  Let us consider the morphism $\Psi_0: \CP^n\rightarrow \PP A_0^*$. $\Psi_0$ is indeed a morphism, but not an embedding: with the coordinates of $U_0$, $\Psi_0(z)=(z_1^2, z_1z_2, z_1z_3,\cdots)$. So it is easy to check that when $p\neq 0$, $d\Phi_0$ is of rank $n$, namely $\Phi_0$ is a local embedding around $p$. Therefore, for any $\xi\in T*'_p$, there exists $s\in A_0$ such that $s(p)=0$ and $\partial s=\xi$. This implies that $\nabla' |_p(A_0)=T_p^{*'}\otimes \ocal(m)_p$, therefore $A_0\cap A_p$ is of codimension $n$ in $A_0$. It now remains to show that $A_0\cap A_p$ is not contained in $B_p$.

\smallskip

 Again, because of the action of $SU(n)\subset SU(n+1)$, we can assume that $$p=[1,x,0,0,\cdots,0],$$where $x$ is a positive number. 

To characterize elements in $A_p$, we consider the unitary map $M_p$ defined by $$M_p(Z_0)=\frac{Z_0+xZ_1}{1+x^2}, M_p(Z_1)=\frac{xZ_0-Z_1}{1+x^2}$$ and $M_p(Z_i)=Z_i$ for $i\neq 0,1$. On can easily see that $M_p$  switch $0$ and $p$. So $M_p^{-1}=M_p$. Now $M_p$ defines an isomorphism, which we still denote by $M_p$, by $M_p(s)(Z)=s(M_p(Z))$ for any $s\in A_0$. Now let $s=aZ_0^{m-3}Z_1^3+bZ_0^{m-2}Z_1^2+Z_0^m$, then by the characterization of $A_0$, it is clear that $s\in A_0$. So $M_p(s)\in A_p$. In terms of the coordinates of $U_0$, one can compute and conclude that 
\begin{eqnarray*}
M_p(s)(z)=\frac{1}{(1+x^2)^m}\{1-ax^3+bx^2+[b(x^2(m-2)-2)\\ +a(3x-(m-3)x^3)+m]xz_1+\dots\}
\end{eqnarray*}
Now we want to show that we can find $a$ and $b$ such that $b(x^2(m-2)-2)\\+a(3x-(m-3)x^3)+m=0$, namely $M_p(s)\in A_0\cap A_p$, and $|b|\neq \frac{m}{2}$, namely $M_p(s)\notin B_p$. But this is very easy to see.

Therefore, for every point $p\in \CP^n$ different from $0$, the intersection $B_p\cap A_0$ is a codimension 1 real subvariety of $A_p\cap A_0$, which is a codimension $n$ subspace of $A_0$. This implies that outside of a measure $0$ subset of $A_0$, every $s\in A_0$ has only isolated non-degenerate critical points, which furthermore implies the theorem.
\end{proof}

\section{Quadrics and generalized Morse theory}\label{sec:quadrics}
Now  let $s\in \textrm{H}^0(\CP^n,\ocal(2))$. We suppose that the zero locus of $s$ is non singular. Then by unitary transformations, we can assume that $s=\sum_{i=0}^na_iZ_i^2$, with $$0< a_0\leq a_1\leq a_2\leq \cdots\leq a_n.$$ So if we let $p_i$ be the point with $Z_i$ the only non-vanishing coordinate. Then the $p_i$'s are critical points of $s$.

 Now we restrict $s$ to $U_0$.  Since $p_0=0$, $s\in A_0$. 
 \begin{theo}
 If the sequence $a_i$ is strictly increasing, then $0$ is the only critical points of $s$ in $U_0$. Moreover, $0$ is non-degenerate and of index $n$ (as a critical point of $|s|^2$)
 \end{theo}
\begin{proof}

Since $f_s=a_0+\sum_{i=1}^na_iz_i^2,  \d f_s=\sum_{i=1}^n2a_iz_id z_i$, and $\d \phi=\frac{\sum_{i=1}^n2\bar{z_i}d z_i}{1+|z|^2}$, to find critical points of $s$ in $U_0$ is the same thing as solving the system of equations $2a_iz_i=f_s\frac{2\bar{z}_i}{1+|z|^2}$ for $1\leq i\leq n$. 

We will use induction on $n$.

When $n=1$, the equation becomes $(1+|z|^2)a_1z_1=(a_0+a_1z^2)\bar{z}$, which is the same as $a_1z=a_0\bar{z}$. But since $a_0< a_1$, the only solution is $0$.

Now assume the statement holds for $n-1$, then one can see first of all that $z_i\neq 0$ for all $1\leq i\leq n$: since if otherwise, the remaining equations will form a system with $n-1$ variables, and by induction there is no solution other than $0$. Therefore we can consider pairs $\frac{z_i}{z_j}$ by taking quotients of the equations in the system.
And we get $\frac{a_i}{a_j}\frac{z_i}{z_j}=\frac{\bar{z}_i}{\bar{z}_j}$. Then since $\frac{a_i}{a_j}$ is real and not $1$, we see that there is no solution other than $0$.
This finish the part for the uniqueness.

To calculate the index, we will use the results of the calculations in section \ref{sec:basic}. 

Since we want to set $f_s(0)=1$, we divide $f_s$ by $a_0$ and the quotient $g_s=1+\sum_{i=1}^nb_iz_i^2$, where $b_i=\frac{a_i}{a_0}$. So by the assumption on $a_i$, we have $1<b_1<b_2<\cdots <b_n$. So by the notation of section \ref{sec:basic},
$$A=\textit{diag}(2ib_1, 2ib_2,\cdots, 2ib_n)$$
Now it is clear that the index of the matrix
\[\left(
\begin{array}{cc}
-nI+\frac{1}{n}A\bar{A} & 0\\
0 & -nI

\end{array}
\right)\] 
is $n$. Also since the above matrix is non singular, $0$ is isolated and non-degenerate as a critical point of $s$

\end{proof}

By similar argument as in the proof of the theorem above, one sees that the $p_i$'s, $0\leq i\leq n$ are all the critical points of $s$, which are all non-degenerate. And the index  at $p_i$ is $n+i$. So we have proved theorem \ref{theo:quadric}.

To prove the corollary, let us first recall the generalized Morse theory by Atiyah and Bott \cite{ab}. 

Given a smooth function $f$ over a  compact smooth manifold $M$. If all of the critical points of $f$ are non-degenerate, then we can form the Morse counting-series
$$M_t(f)=\sum_pt^{\lambda_p(f)},  df(p)=0$$
where the sum ranges over the critical points of $f$, and $\lambda_p(f)$ is the index of $f$ at $p$. And such a function $f$ is called a Morse function.

Now let $P_t(M)=\sum t^i \dim H^i(M, \R)$ be the Poincar\'e series for $M$. Then we have the following Morse inequalities: There exists a polynomial $R(t)$ with non-negative coefficients, such that $$M_t(f)-P_t(M)=(1+t)R(t)$$

The generalized Morse theory replace non-degenerate critical points by non-degenerate critical manifolds. More precisely, a connected submanifold $N\subset M$ is called a non-degenerate critical manifold for $f$ if and only if $df\equiv 0$ along $N$ and $H_Nf$ is non-degenerate on the normal bundle $\nu(N)$ of $N$, where $H_Nf$ is the Hessian.  If this is the case, $H_Nf$ defines a canonical self-adjoint endomorphism: $A_N:\nu(N)\rightarrow \nu(N)$ by the formula $(A_Nx,y)=H_Nf(x,y),$ for any $x,y\in \nu (N)$. Then the index $\lambda_N$ of $f$ at $N$ is defined to be the dimension of the subbundle $\nu^-(N)\subset \nu(N)$, which is defined to be the vectors which are eigen-vectors of $A_N$ with negative eigenvalues. 

Assume all the critical manifolds $N$ are non-degenerate, $f$ is then called a generalized Morse function. If furthermore all $N$'s are orientable (in \cite{ab}, also the non-orientable case is also considered, but for the use of this article, we only need the orientable case ), we can also define the Morse counting-series by 
$$M_t(f, N)=\sum t^{\lambda_N}P_t(N)$$
Then the Morse inequality persists.

Now the zero locus $D$ of a non-degenerate  $s\in H^0(\CP^n,\ocal (2)))$ is smooth and therefore a non-degenerate critical manifold with index $0$ of the function $|s|^2$(\cite{bo}). Therefore for a general $s$ as in the assumption of theorem \ref{theo:quadric}, $|s|^2$ is a generalized Morse function. And we can apply the Morse inequality to get 
$$P_t(D)+\sum_{i=0}^{n-1}t^{n+i}-P_t(\CP^n)=(1+t)R(t)$$
Now by Lefschetz hyperplane theorem and Poincar\'e duality, when $n=2k+1$
$$h^{2k}(D)t^{2k}+\sum_{i=0}^nt^{n+i}=(1+t)R(t)$$
where $h^{2k}(D)=\dim H^{2k}(D,\R)$.

Letting $t=-1$, we see that $h^{2k}(D)=2$.

So the corollary is proved.

\begin{rem}
It is interesting to give a slightly different proof of this corollary here.

A very similar argument in in the proof of theorem \ref{theo:quadric} can be used to prove that in the special case when all the $a_i$'s are the same, the submanifold $N=\{\textit{all the } Z_i \textit{ are real}\}$, which is just an $\R\PP^n$, is a non-degenerate critical manifold of index $n$ of $|s|^2$. Recall that the Poincar\'e polynomial of $\R\PP^n$ is $1+t^n$ when $n$ is odd. Then another application of the generalized Morese inequality will result in the same conclusion.

\end{rem}

\section{Spherical Gauss-Lucas theorem}\label{sec:gauss}

Now we come to the proof of theorem \ref{theo:gauss}.

The reason that we exclude the case when $n=1$ is simple: the zero of $s$ is not a critical point of $s$, although it is a critical point of $|s|^2$. When $n=2$, by theorem \ref{sec:quadrics}, $s$ has only two critical points. One is the middle point of the minimizing geodesic segment connecting the two zeroes, and it is of index $1$; the other one is the opposite point of the index 1 critical point, and it is of index $2$. So the theorem holds for these two cases.

When $n\geq 3$, the situation gets much more complicated. To prove the theorem, we need to first distinguish the points that are inside the convex hull and that are outside. Given a convex polygon $P\subset \CP^1$ that is contained in an open hemisphere. Let $P_{\infty}$ be the opposite polygon of $P$.  We denote by $v_i$ the vertices of $P$. Now given a point $q$ in $\CP^1$, we assume that $q$ is not one of the opposite points of the $v_i$'s. Then for each $v_i$, there exits an unique minimizing geodesic $l_i(q)$ from $q$ to $v_i$. The tangent vector of $l_i(q)$ at $q$ generates a ray $r_i(q)$ in the tangent space of $\CP^1$ at $q$.

\begin{lem}
With the notations as above. A point $q\in \CP^1$ is contained in the inside of $P$ or the inside of $P_{\infty}$ if and only if the convex cone $K_{P}(q)$ generated by the $r_i$'s is the whole tangent space $T_q=T_{\CP^1,q}$, namely there does not exist a line through the origin in $T_{\CP^1,p}$ such that all the rays $r_i$ lie on one side of the line. 
\end{lem}

\begin{proof}
When $q$ is contained in the inside of $P$, it is clear that $K_P(q)=T_q$. When $q$ is contained in the side of $P_{\infty}$, we have $K_P(q)=-K_{P_{\infty}}(q)$. But since in this case $K_{P_{\infty}}(q)=T_q$, we also have $K_P(q)=T_q$.

Now we assume that $q$ is contained in neither the inside of $P$ nor the inside of $P_{\infty}$. If $q$ lies on one edge of $P$ or $P_{\infty}$, then the great circle containing this edge divide the sphere into two hemisphere and $P$ lies in the closure of one. So in this case $K_P(q)$ is a half plane in $T_q$. If $q$ is also not contained in any edges of $P$ or $P_{\infty}$, then we claim that $K_P(q)\cap K_{P_{\infty}}(q)=0$. But this is not hard to see, since if we extend the minimizing geodesic  from $q$ to a point in $P$, it will not touch $P_{\infty}$ before arriving at the opposite point of $q$. 

Since we always have $K_P(q)=-K_{P_{\infty}}(q)$, this will imply that $K_{P}(q)\neq T_q$ .
\end{proof}

\begin{rem}
\begin{enumerate}
\item It follows from the proof that $q$, which is not one of the vertices, is contained in an edge of $P$ or $P_{\infty}$ if and only $K_P(q)$ equals an half plane in $T_q$. 

\item In the case when $P$ is just a minimizing geodesic segment, $q$ is contained in the interior of $P$ or in the interior of $P_{\infty}$ if and only if $K_P(q)$ is a whole line (not a ray) in $T_q$. 

\item  If $P$ is convex polygon (a minimizing geodesic segment), and $q$ is one of the vertices (end points)of $P_{\infty}$, then we define $K_P(q)$ to be  the convex cone generated by the $r_i(q) $ for those $i$'s none of which is the opposite point of $q$. So in this caee, $K_P(q)$   is contained in but not equal to a half plane of $T_q$, when $P$ is a convex polygon. And if $P$ is a minimizing geodesic segment,  then $K_P(q)$ is just a ray in $T_q$

\end{enumerate}

\end{rem}
Recall we denote by $0$ the point $[1,0]$. Now if $s\in A_0$, let $x_i, 1\leq i\leq n$ be the zeroes of $s$, which are contained in an open hemisphere. Then by the remark above, we consider the rays $r_i(q)$ in $T_0$ which is generated by the tangent vectors at $0$ of the minimizing geodesics from $0$ to those  $x_i$'s which are not $\infty=[0,1]$.  But since the polynomial $f_s$ has no degree $1$ term, which implies that $\sum \frac{1}{x_i}=0$. This equation implies that if  these $x_i$'s  lie on one great circle, then $K_P(0)$ is a whole line in $T_q$, and that if they generate a convex polygon $P$, then $K_P(0)=T_q$. In both cases, $0$ is contained in the interior of $P$ or $P_{\infty}$.

Now we proceed to prove that if $P$ is not a single point, then both $C_{s,P}$ and $C_{s,P_{\infty}}$ are non-empty.

Our idea is to consider the gradient field of $\log |s|^2$. If we consider $s$ as a homogeneous polynomial of degree $n$, then $s$ is a product of linear terms $s=\prod_{i=1}^n (W^{(i)}_0Z_0+W^{(i)}_1Z_1)$. Therefore 
$$|s|^2=\prod \frac{|(W^{(i)}_0Z_0+W^{(i)}_1Z_1|^2}{|Z_0|^2+|Z_1|^2}$$
and $\log |s|^2=\sum_{i=1}^n \log |L_i|^2$, where $L_i=W^{(i)}_0Z_0+W^{(i)}_1Z_1$ is considered as an element in $H^0(\CP^1,\ocal(1))$. And it is easy to observe that the union of the zeroes of the $L_i$'s is the zero set of $s$. We denote by $G(s)$ the gradient field of $\log |s|^2$ , and by $G(L_i)$ the gradient of $ \log |L_i|^2$ , then we have 
$$G(s)=\sum_{i=1}^n G(L_i)$$
The gradient $ G(L_i)$ is very clear to us: Let $x_i$ denote the zero of $L_i$ and $y_i$ the opposite point of $x_i$, then $ G(L_i)$ at $y_i$ is $0$, and at a point $q$ different from $x_i$ and $y_i$,  $ G(L_i)(q)$ is tangent to the geodesic determined by the three points $x_i$,  $q$ and $y_i$, and it points to $y_i$.

\smallskip

Now if $P_{\infty}$ is a minimizing geodesic segment, then at the two end points $G(s)$ is not $0$ and points towards the interior of  $P_{\infty}$. And at a interior point of $P_{\infty}$, $G(s)$, if not $0$, is tangent to $P_{\infty}$. One sees from this observation that the interior of $P_{\infty}$ contains at least one critical point which is a local maximum of  $\log |s|^2$, hence of index $2$ if non-degenerate. Meanwhile $P$ is also a minimizing geodesic segment. Although $G(s)$ is not defined at  the zeroes $x_i$, $G_s$ is tangent to $P$ and points towards $y_i$ in a small neighborhood of each $x_i$. This is because the norm of $G(L_i)$ goes to $\infty$ when getting close to $x_i$. Therefore we conclude that between each pair of adjacent zeroes, there exists at least one critical point of $\log |s|^2$.

\smallskip

What remains to prove is the general situation when $P$ is a convex polygon. We will refer to $G(L_i)$ as the "contribution of $x_i$". On each edge of $P_{\infty}$, since the contributions of the zeroes on the edge are along the edge, we conclude that $G(s)$ is non vanishing at each point on the boundary of $P_{\infty}$ and points towards the interior of $P_{\infty}$. Therefore the interior of $P_{\infty}$ contains at least one critical point which is a local maximum of  $\log |s|^2$, hence of index $2$ if non-degenerate.  

\smallskip

If we furthermore assume that all the critical points  $|s|^2$ are non-degenerate, which by theorem \ref{theo:nondeg} is the general case, we want to prove that $C_{s,P}$ is also non-empty and contained in the interior $P$. For this we need to consider the gradient of $|s|^2$, which vanishes at the zeroes of $s$ and is a positive multiple of $G(s)$ at a general point. So we can choose a smooth Jordan curve $B$ around $P_{\infty}$, which divide $\CP^1$ into two components, such that  the gradient $grad(|s|^2)$ does not vanish on $B$, and points towards the interior of the component containing $P_{\infty}$. With all these conditions, we look at the winding number of $grad(|s|^2)$ along $B$ with the orientation that makes the component containing $P$ the inside. The winding number is $1$ by Poincar\'e theorem, which implies that $|s|^2$ has at least $n-1$ critical points of index $1$ in the interior of $P$. Therefore $C_{s,P}$ is non-empty and contained in the interior of $P$. We have completed the proof of theorem \ref{theo:gauss}.

\end{document}